\documentclass[preprint,1p,times]{elsarticle}
\usepackage{amssymb}
\journal{$\;$}

\begin{document}

\newtheorem{theorem}{Theorem}[section]
\newtheorem{lemma}{Lemma}[section]
\newtheorem{remark}{Remark}[section]
\newtheorem{example}{Example}[section]
\newdefinition{definition}{Definition}[section]
\newproof{proof}{Proof}
\renewcommand{\theequation}{\thesection.\arabic{equation}}

\begin{frontmatter}


\title{Range of certain convolution operators  and reconstruction from local averages}
\author{P. Devaraj}

\address{School of Mathematics, Indian Institute of Science Education and Research,\\ Thiruvananthapuram-695551, Kerala, India.\\
e-mail: devarajp@iisertvm.ac.in}

\begin{abstract}
For a   compactly supported absolutely continuous measure $\mu$    on ${\mathbb{R}}^2$ having  a density function equal to
a finite linear combination of indicator functions of rectangles $\left[a_{i}, b_{i}\right]\times \left[c_{i}, d_{i}\right],$  we analyse the range of the convolution operator $C_{\mu}:C({\mathbb{R}}^2)\rightarrow C({\mathbb{R}}^2)$ defined by $C_{\mu}(f)=f\star\mu,$ where
  $(f\star \mu)(x,y)=\int_{{\mathbb{R}}^2}f(x-s,y-t)d\mu.$  It  is shown that $C_{\mu}$  maps the space of all continuous  functions $C({\mathbb{R}}^2)$ onto the space $C^{2*}({\mathbb{R}}^2)=\{f:{\mathbb{R}}^2\rightarrow {\mathbb{C}}:\frac{\partial^2 f}{\partial x \partial y},\frac{\partial^2 f}{\partial y \partial x}\in C({\mathbb{R}}^2)\}$
  provided the density function of $\mu$  satisfies certain conditions. 
\end{abstract}


\end{frontmatter}






 \section{Introduction}
Many linear translation invariant  systems like atmospheric imaging and image reconstruction are modeled using convolution type of integral equations.
For instance, when a ground-based telescope is used to take an astronomical image, serious blurring
can occur. This is  due to the scattering of light waves when they pass through the atmosphere in the presence of  atmospheric
turbulence. A simple model\cite{Pat} of this image blurring can be written
as
\begin{equation}
T(f)(x, y) = g(x, y), (x, y)\in \Omega,
\label{eqn1.1}
\end{equation}
where $f(x, y)$ represents the gray-values of the true image at the two-dimensional
location $(x, y)$  and
 $g(x, y)$ denotes the gray-values of the observed image (the
“data”, as collected using the telescope), $\Omega$
 is the (x, y)-region of the image, and the
process of blurring or scattering is described via the bounded linear operator T,

$$T(f)(x, y) =\int_{\Omega}f(x-s,y-t) k(s,t) ds dt, (x, y)\in \Omega.$$
The kernel $k$ is known as a point spread function and is generally assumed to be given.

 In these scenarios equation \ref{eqn1.1}  can be written in a most general form as $(f\star \mu)(x,y)=g(x,y),$ where $$(f\star \mu)(x,y)=\int_{-\infty}^{\infty}\int_{-\infty}^{\infty} f(x-s,y-t)d\mu.$$ We consider the problem of reconstructing $f$ from $g$ and $\mu$ when the density function of $\mu$ is a finite linear combination of indicator functions of rectangles in  ${\mathbb{R}}^2.$ That is we are interested in characterizing the range of the operator  $$C_{\mu}:C({\mathbb{R}}^2)\rightarrow C({\mathbb{R}}^2)$$ defined by $C_{\mu}(f)=f\star \mu.$
 We explicitly  construct a solution $f$ for every $g$ in $C^{2*}({\mathbb{R}}^2).$
 The functions in the kernel of this  convolution operator $C_{\mu}$ for a given $\mu$ are called mean-periodic functions. The characterization of the kernel has been analysed on various groups \cite{bag,ber1,Del,deva,Ehr1,Ehr2,kah,sch,sze,wei}.

Erich Novak and Inder K. Rana\cite{rana1}, have    analyzed the range of $C_{\mu}$ when the density function of the convolver $\mu$ is the indicator function of an
interval $[-a,a]$ on ${\mathbb{R}}$ and the same has been extended in  \cite{deva1,deva3} for other different types of the convolvers on the group ${\mathbb{R}}$.  Another special case  namely finitely supported measures $\mu$ on ${\mathbb{R}}^2$,  is analysed in \cite{deva2}.

 Radon \cite{rad} has analysed a related problem and showed  that a function $f\in C_{c}({\mathbb{R}}^2)$ can be recovered  explicitly from its integrals over all straight lines in ${\mathbb{R}}^2.$ The inversion formula of Radon and its generalizations form the core of computerized tomography.  Radon's inversion has been generalized to higher-dimensional Euclidean spaces and to various symmetric spaces  by Helgason\cite{hel} and others.

 Another closely  related problem{\cite{pom} was posed by Pompeiu  for n=2: Is it possible to recover a continuous function $f\in C({\mathbb{R}}^2)$ from its integrals over all balls of a fixed radius $r>0?$  The one ball transform

 $$B_{r}f(x) := \int_{|y|\le r} f(x+y)dy$$
 is not one-to-one on $C({\mathbb{R}}^n)$ \cite{bro, cha}. But still the convolution equation
 \begin{equation}
 B_{r}f=g
 \label{eqn1}
 \end{equation}
 may have solution $f$ and is unique if $f$ is assumed to satisfy some priori conditions.
 Explicit inversion formula  is known only for $n=1$\cite{rana1}. A  more general problem is that of solving the convolution equation of the following type
  \begin{equation}
 \sum_{i=1}^{k}\alpha_{i}B_{r_{i}}f=g,
 \label{eqn2}
 \end{equation}
where  $\alpha_{i}$ are complex constants.The above convolution operator  is analysed in \cite{deva3} over real line.

Novak\cite{nov} has obtained  an inversion formula for $f$ from its integrals namely,

 $$T_{n}f(x)=2^{-n}\int_{x+Q}f(y)dy,$$
 where $Q=[-1,1]^{n}.$

 In this paper, we analyse the range of $C_{\mu}$ when the density function of $\mu$ is a finite linear combinations indicator functions of rectangles in ${\mathbb{R}}^2$ satisfying certain conditions.

\section{Notations and Preliminaries}
\setcounter{equation}{0}
 By  a convex cone in ${\mathbb{R}}^2$ we mean a nonempty  subset $C$ of ${\mathbb{R}}^2$
 such that  $(x_1,y_{1}),(x_2,y_{2}),\ldots,(x_{k},y_{k})\in C$ and $\alpha_i\ge 0$  implies
$(\sum_{i=1}^{k}\alpha_i (x_i,y_{i}))\in C.$
 A convex cone $C$ satisfying  $C\cap(-C)=\{0\}$ is called a pointed convex cone. If a convex cone is generated by two nonzero vectors $u$ and $v$ in  ${\mathbb{R}}^2$,
 then its bisector is the ray $\displaystyle{\left\{t\left(\frac{u}{2\|u|}+\frac{v}{2\|v\|}\right): t\ge 0\right\}}.$

A line $L(v,\alpha)$ determined by a vector $v=(v_{1},v_{2})$ and a scalar $\alpha$ is the  set $\{(x,y)\in {\mathbb{R}}^2:v_{1}x+v_{2}y=\alpha\}.$ The left-half and the right-half spaces determined by a line $L(v,\alpha)$ are the sets $HS_{l}(v,\alpha)=\{(x,y)\in {\mathbb{R}}^2: v_{1}x+v_{2}y\le \alpha\}$ and $HS_{r}(v,\alpha)=\{(x,y)\in {\mathbb{R}}^2: v_{1}x+v_{2}y\ge \alpha\}$ respectively.

\centerline{$\displaystyle{\pi_{i}(x,y):=\left\{\begin{array}{lll} x&\mbox{if}&i=1\\
y&\mbox{if}&i=2. \end{array} \right.\\
}$} denotes   the coordinate-wise projections and the following notations are used for the translated  quadrants  of the plane:
\begin{eqnarray*}
Q_{1}(\alpha,\beta) &=&\{(x,y)\in {\mathbb{R}}^2:x\ge \alpha, y \ge \beta\},\\
Q_{2}(\alpha,\beta) &=&\{(x,y)\in {\mathbb{R}}^2:x\le \alpha, y \ge \beta\},\\
Q_{3}(\alpha,\beta) &=&\{(x,y)\in {\mathbb{R}}^2:x\le \alpha, y \le \beta\},\\
Q_{4}(\alpha,\beta) &=&\{(x,y)\in {\mathbb{R}}^2:x\ge \alpha, y \le \beta\}.\\
\end{eqnarray*}

For the rectangles, $\displaystyle{R_{i}=[a_{i},b_{i}] \times [c_{i},d_{i}]}$, we denote
$$ M_{i}=\frac{a_{i}+b_{i}}{2},\;
N_{i}=\frac{c_{i}+d_{i}}{2},\;
S_{i}=\frac{b_{i}-a_{i}}{2}\mbox{ and }
T_{i}=\frac{d_{i}-c_{i}}{2}.$$

We use the notation $C^{m}({\mathbb{R}}^{2})$ for $m$ times continuously differentiable functions. i.e., $C^{m}({\mathbb{R}}^2)=\{f:{\mathbb{R}}^{2}\rightarrow {\mathbb{C}}:\frac{\partial^{i+j}f}{\partial x^{i}\partial y^{j}} \in C({\mathbb{R}}^2),0\le i,j\le m, 0\le i+j\leq m\}.$ Also
$C^{2*}({\mathbb{R}}^2)=\{f:{\mathbb{R}}^2\rightarrow {\mathbb{C}}:\frac{\partial^2 f}{\partial x \partial y},\frac{\partial^2 f}{\partial y \partial x}\in C({\mathbb{R}}^2)\}.$

\begin{definition}
A  Borel measure $\mu$ on ${\mathbb{R}}^2$ is called a
discrete  Borel measure, if  $supp(\mu)\cap ([-n,n]\times [-n,n])$ is a finite subset of ${\mathbb{R}}^{2}$ for every $n\in {\mathbb{N}}.$
 The set of all  discrete Borel measures on ${\mathbb{R}}^2$ is denoted by $M_{d}({\mathbb{R}}^2 ).$
\end{definition}

If  $\mu$ is a compactly supported discrete Borel measure on ${\mathbb{R}}^2,$ then there exists a distinct set of vectors  $(x_1,y_{1}), (x_2,y_{2}),\ldots,(x_k,y_{k})\in {\mathbb{R}}^2$ and nonzero complex constants $\alpha_1,\alpha_2,\ldots,\alpha_k$ such that $\displaystyle{\mu(E)=\sum_{i=1}^k
\alpha_i\delta_{(x_i,y_{i})}(E)}$ for every Borel set $E$ of ${\mathbb{R}}^2.$

 We denote by $M_{cd}({\mathbb{R}}^2)$ the set of all
compactly supported  discrete Borel measures on ${\mathbb{R}}^2.$  The set of all compactly supported regular Borel
measures on ${\mathbb{R}}^2$  is denoted by $M_{c}({\mathbb{R}}^2).$
\begin{definition}
The convolution of a function $f\in C({\mathbb{R}}^2)$  with a compactly supported $\mu\in M_{c}({\mathbb{R}}^2)$ is defined as

\centerline{$\displaystyle{(f\star\mu)(x,y)=\int_{\mathbb{R}^2}f(x-s,y-t)d\mu}.$}
\end{definition}

\noindent The above convolution can be written as

\centerline{ $\displaystyle{(f\star\mu)(x,y)=\sum_{i=1}^{k} \alpha_i f(x-x_i,y-y_{i})},$  }
\noindent
when
$\displaystyle{\mu=\sum_{i=1}^{k} \alpha_i\delta_{(x_i,y_{i})}}\in M_{cd}({\mathbb{R}}^2).$

\noindent Also, when  the density function of $\mu$ is  a finite linear combination $\displaystyle{\sum_{i=1}^{k}\alpha_{i}\chi_{[a_{i},b_{i}]\times [c_{i},d_{i}]}}$ of indicator functions, the convolution becomes

\centerline{$\displaystyle{(f\star \mu)(x,y)=\sum_{i=1}^{k}\alpha_{i}\int_{a_{i}}^{b_{i}}\int_{c_{i}}^{d_{i}}f(x-s,y-t)dtds.}$}

For  a discrete measure, the convolution $f\star \mu$ is not always defined, for instance if $\mu=\sum_{i=1}^\infty \left(\frac{1}{i^2}\right)\delta_{(-i,0)}, {\mbox and } f(x,y)=x,$
then $(f \star \mu)(0,0)=\sum_{i=1}^\infty \frac{1}{i},$ which is not defined.
 In order to define the convolution in such cases  we need to restrict it to a  suitable class of continuous functions and a corresponding class of discrete measures.   These are obtained in the following lemma:

\begin{lemma} Let $f\in C({\mathbb{R}}^{2})$ and $\mu\in M_{d}({\mathbb{R}}^{2}).$
Let $C\subset {\mathbb{R}}^2$ be a pointed convex cone having a bisector $v=(v_{1},v_{2})\in {\mathbb{R}}^2 $ such that $C\subset H_{r}(v,0)$.
If either $supp(f)\subset H_{r}(v,\alpha)$ and $supp(\mu)\subset C$  for some $\alpha\in {\mathbb{R}}$ or  $supp(f)\subset H_{l}(v,\alpha)$ and $supp(\mu)\subset -C$  for some $\alpha\in {\mathbb{R}},$ then $(f\star\mu)(x,y)$ is well defined for every $(x, y)\in {\mathbb{R}}^{2}.$ Moreover, if $f\in C^{m}({\mathbb{R}}^2)$, then $\frac{\partial^{m}(f\star\mu)} {\partial x^p \partial y^q}(x,y)=( \frac{\partial^{m}f} {\partial x^p \partial y^q}\star\mu)(x,y).$
\label{lemnew1}
\end{lemma}
\begin{proof}Suppose that  $supp(f)\subset H_{r}(v,\alpha)$ and $supp(\mu)\subset C.$

 \noindent Then \begin{eqnarray*}(f\star\mu)(x,y)&=&\int_{-\infty}^{\infty}\int_{-\infty}^{\infty}f(x-s,y-t)d\mu.\\
 &=&\int\int_{C\cap H_l(v,\alpha_{0})}f(x-s,y-t)d\mu,
 \end{eqnarray*}
 where $\alpha_{0}=xv_{1}+yv_{2}-\alpha.$

\vskip 1em
\noindent
 It is easy to see that $C\cap H_l(v,\alpha_{0})$ is a compact set. Otherwise, $C\cap H_l(v,\alpha_{0})$ will be unbounded  and hence $C$ contains every point along a direction $u$ perpendicular to $v$. As $v$ is a bisector, $C$
 contains every vector along $-u.$ This will contradict the fact that $C$ is pointed. Therefore, for every $(x,y)\in {\mathbb{R}}^2,$ $supp(\mu)\cap C\cap H_{l}(v,\alpha_{0})$ is a finite set and hence  $(f\star\mu)(x,y)$ is a finite sum.

Further, one can easily check that for every  $(x,y)\in {\mathbb{R}}^2,$ there exist $\beta\ge \alpha_{0}=xv_{1}+yv_{2}-\alpha$ in ${\mathbb{R}}$ and $\rho>0$ such that for $|h|<\rho,$
$$supp(\mu)\cap C\cap H_r(v,\beta-|h|)\cap H_{l}(v,\beta+|h|)=\emptyset.$$ Therefore we can write,
\begin{eqnarray*}
 (f\star\mu)(x,y)&=&\sum_{i=1}^n \alpha_{i}f(x-x_{i},y-y_{i}),\\
 (f\star\mu)(x+h,y)&=&\sum_{i=1}^n \alpha_{i}f(x+h-x_{i},y-y_{i}),\\
\end{eqnarray*}
for $|h|<\rho$ and  hence
 $$\lim_{h\rightarrow 0}\frac{(f\star \mu)(x+h,y)-(f\star\mu)(x,y)}{h}=\sum_{i=1}^{n}\alpha_{i}\frac{\partial f}{\partial x}(x-x_{i},y-y_{i}).$$
By repeated use of similar arguments we obtain,
$$\frac{\partial^{m}(f\star\mu)} {\partial x^p \partial y^q}(x,y)=\left( \frac{\partial^{m}f} {\partial x^p \partial y^q}\star\mu\right)(x,y).$$
 \vskip 1em
\noindent
 The proof for the case $supp(f)\subset H_{l}(v,\alpha)$ and $supp(\mu)\subset -C$ can be proved in the similar way.

\hfill{$\square$}

\end{proof}

 \begin{lemma}
Let $C$ be a pointed convex cone with bisector $v$ such that $C\subset H_{r}(v,0).$ Then  every $g\in C^{m}({\mathbb{R}}^2)$  can be written as $g=g_{1}+g_{2},$ where $g_{1},g_{2}\in C^{m}({\mathbb{R}}^2),$  $supp(g_{1})\subset   H_{r}(v,\alpha_{1})$ and $supp(g_{2})\subset  H_{l}(v,\alpha_{2})$ for some
$\alpha_{1},\alpha_{2}\in {\mathbb{R}}.$
\label{lemma2.4}
\end{lemma}
\begin{proof}
For $\epsilon>0,$ choose $\phi_{\epsilon}\in C_{c}^{m}({\mathbb{R}}^{2})$ such that $supp(\phi_{\epsilon})\subset H_{r}(v,-\epsilon)\cap H_l(v,\epsilon)$ and  $$\int_{{\mathbb{R}}^2}\phi_{\epsilon}(x,y)dxdy=1.$$

\noindent Define $h_1,h_2\in C({\mathbb{R}}^2)$ by

$$\displaystyle{h_1(x,y):=\left\{\begin{array}{lll} 0&\mbox{if}&(x,y) \in H_{l}(v,-n)\\
\frac{n+v_{1}x+v_{2}y}{2n}&\mbox{if}&(x,y) \in H_{r}(v,-n)\cap H_{l}(v,n)\\
1&\mbox{if}&(x,y)\in H_{r}(v,n),
\end{array}
\right. }$$

$$\displaystyle{h_2(x,y):=\left\{\begin{array}{lll} 1&\mbox{if}&(x,y) \in H_{l}(v,-n)\\
\frac{n-(v_{1}x+v_{2}y)}{2n}&\mbox{if}&(x,y) \in H_{r}(v,-n)\cap H_{l}(v,n)\\
0&\mbox{if}&(x,y)\in H_{l}(v,n).
\end{array}
\right. }$$

\noindent It is simple to check that $h_{1}(x,y)+ h_{2}(x,y)=1$  for all $(x,y)\in {\mathbb{R}}^2.$

\noindent Define
\begin{eqnarray*}
g_1(x,y)&=&g(x,y)(h_{1}\star\phi_{\epsilon})(x,y),\\
g_2(x,y)&=&g(x,y) (h_{2}\star\phi_{\epsilon})(x,y).
\end{eqnarray*}

 \noindent Then $g_{1},g_{2}\in C^{m}({\mathbb{R}}^2)$
 and  $g_{1}+g_{2}=g.$ Also
\begin{eqnarray*}
 Supp(g_{1})&\subset & H_{r}(v,-n-2\epsilon),\\
  Supp(g_{2})&\subset & H_{l}(v,n+2\epsilon).\\
    \end{eqnarray*}
\end{proof} \hfill{$\Box$}

It may be noted that in the above proof if we take $g,\phi_{\epsilon}\in C^{2*}({\mathbb{R}}^2),$ then we get $g_{1},g_{2}\in C^{2*}({\mathbb{R}}^2).$

\begin{lemma}
Let $\mu_{1},\mu_{2}\in M_{d}({\mathbb{R}}^2)$ be two nonzero discrete measures such that $supp(\mu_{1})\subset Q_{1}(\alpha_{1},\beta_{1})$  and $supp(\mu_{2})\subset Q_{3}(\alpha_{2},\beta_{2}).$ Then there exist $(x^{*},y^{*})\in supp(\mu_{1})$ and $(x^{**},y^{**})\in supp(\mu_{2})$ and a pointed convex cone  $C$ with bisector $v$ such that
\begin{enumerate}\item $C\subset H_{r}(v,0),$ $Q_{1}(0,0)\subset C,$ $ Q_{3}(0,0)\subset -C,$
 \item $\mu_{1}=c^{*}\delta_{(x^{*},y^{*})}\star (\delta_{0}+\eta_{1}),$  $supp(\eta_{1})\subset C\cap H_{r}(v,\alpha')$ for some $\alpha'>0,$
  \item $\mu_{2}=c^{**}\delta_{(x^{**},y^{**})}\star (\delta_{0}+\eta_{2})$ and $supp(\eta_{2})\subset (-C)\cap H_{l}(v,-\beta')$ for some $\beta'>0.$
      \end{enumerate}
\label{lemma2.5}
\end{lemma}
\begin{proof}
Define
\begin{eqnarray*}
x^{*}&=&Min\{\pi_{1}(x,y):(x,y)\in supp(\mu_{1})\},\\
y^{*}&=&Min \{\pi_{2}(x,y):\pi_{1}(x,y)=x^{*},(x,y)\in supp(\mu_{1})\}.
\end{eqnarray*}

As $supp(\mu_{1})\cap \overline{B}(0,r)$ is a finite set for every $r$,
$(x^{*},y^{*})\in supp(\mu_{1}).$
Therefore,

  $$supp(\mu_{1}\star \delta_{-(x^{*},y^{*})})\subset Q_{1}(0,\beta_{1}-y^{*}).$$

Similarly, define
\begin{eqnarray*}
x^{**}&=&Max\{\pi_{1}(x,y):(x,y)\in supp(\mu_{2})\},\\
y^{**}&=&Max \{\pi_{2}(x,y):\pi_{1}(x,y)=x^{**},(x,y)\in supp(\mu_{2})\}.
\end{eqnarray*}
Then $(x^{**},y^{**})\in supp(\mu_{2})$   and  $supp(\mu_{2}\star \delta_{-(x^{**},y^{**})})\subset Q_{3}(0,\beta_{2}-y^{**}).$ Choose $r>0$ such that $r>2 Max \{|\beta_{1}-y^{*}|,|\beta_{2}-y^{**}|\}.$

Define
\begin{eqnarray*}
A&=&\left\{(x,y)\in supp(\mu_{1}): \pi_{1}(x-x^{*},y-y^{*})\neq 0,  (x-x^{*},y-y^{*})\in \overline{B}(0,r)\right\},\\
B&=&\left\{(x,y)\in supp(\mu_{2}): \pi_{1}(x-x^{**},y-y^{**})\neq 0,  (x-x^{**},y-y^{**})\in \overline{B}(0,r)\right\},\\
s_{1}&=&Max \left\{0,\left|\frac{\pi_{2}(x-x^{*},y-y^{*})}{\pi_{1}(x-x^{*},y-y^{*})}\right|,\left|\frac{\beta_{1}-y^{*}}{\sqrt{r^2-(\beta_{1}-y^{*})^2}}\right|
:(x,y)\in A\right\},\\
s_{2}&=&Max \left\{0,\left|\frac{\pi_{2}(x-x^{**},y-y^{**})}{\pi_{1}(x-x^{**},y-y^{**})}\right|,\left|\frac{\beta_{2}-y^{**}}{\sqrt{r^2-(\beta_{1}-y^{**})^2}}\right|: (x,y)\in B\right\}, {\mbox and}\\
s&=&Max\{s_{1},s_{2}\}.
\end{eqnarray*}
If $s=0,$ take $C=Q_{1}(0,0)$ and $\displaystyle{v=\left(\frac{1}{\sqrt{2}},\frac{1}{\sqrt{2}}\right).}$

If $s>0$, then take $C$ to be the pointed convex cone generated by the vectors $(0,1)$ and $\displaystyle{\left(\frac{1}{\sqrt{1+s^2}},\frac{-s}{\sqrt{1+s^2}}\right).}$
This $C$ has bisector $\displaystyle{v=\left(\frac{1}{2\sqrt{1+s^2}},\frac{\sqrt{1+s^2}-s}{2\sqrt{1+s^2}}\right)}.$
Then $C\subset H_{r}(v,0).$ Clearly $$(x-x^{*},y-y^{*})\in C\cap H_{r}(v,\alpha) \mbox{ for every  }(x,y)\in supp(\mu_{1})\setminus \{(x^{*},y^{*})\}$$ for some $\alpha>0$ and  $$(x-x^{**},y-y^{**})\in (-C)\cap H_{l}(v,-\beta)
\mbox{ for every }(x,y)\in supp(\mu_{2})\setminus\{(x^{**},y^{**})\}$$ for some $\beta>0.$
Therefore we conclude that $$\mu_{1}=c^{*}\delta_{(x^{*},y^{*})}\star (\delta_{0}+\eta_{1}) \mbox{ and } \mu_{2}=c^{**}\delta_{(x^{**},y^{**})}\star (\delta_{0}+\eta_{2}),$$ where  $$supp(\eta_{1})\subset C\cap H_{r}(v,\alpha) \mbox{ and } supp(\eta_{2})\subset (-C)\cap H_{l}(v,-\beta).$$
\end{proof}
\hfil{       \hspace*{4in}$\Box$}

\begin{remark}
 Let $\mu_{1}$ and $\mu_{2}$ be two nonzero discrete measures  such that $supp(\mu_{1})\subset Q_{2}(\alpha_{1},\beta_{1})$ and $supp(\mu_{2})\subset Q_{4}(\alpha_{2},\beta_{2}).$ Then,  along the lines of the above proof, it can be shown by making appropriate changes that the conclusions similar to the above lemma hold.
\end{remark}

\section{Existence Theorem}
\setcounter{equation}{0}

We  consider the following  two broad situations on the rectangular regions  $R_{i}=\left[a_{i},b_{i}\right]\times \left[c_{i},d_{i}\right]$:
\begin{eqnarray}
\label{cond1}
a_{i}<a_{i+1},b_{i}<b_{i+1}, c_{i}<c_{i+1}, d_{i}<d_{i+1}\\
\label{cond2}
a_{i}>a_{i+1},b_{i}>b_{i+1}, c_{i}<c_{i+1}, d_{i}<d_{i+1}
\end{eqnarray}

The set of conditions (\ref{cond1}) and (\ref{cond2})  essentially mean respectively that $R_{i+1}\subset \mbox{Interior}(Q_{1}(a_{i},c_{i}))$ and $ R_{i+1}\subset \mbox{Interior}(Q_{2}(b_{i},c_{i}))$ and $R_{i}$ is not a subset of $R_{j}$ for $i\neq j$.

\begin{theorem}[Main Theorem]Let $0\neq \mu\in M_{c}({\mathbb{R}}^2)$ be absolutely continuous with a density function equal to a finite linear combination of indicator functions of rectangles $\left[a_{i},b_{i}\right]\times \left[c_{i},d_{i}\right]$ in ${\mathbb{R}}^2$ and $g\in C^{2*}({\mathbb{R}}^2)$. If the conditions (\ref{cond1}) hold  for every $i$ or conditions (\ref{cond2}) hold  for every $i$, then
there exists $f \in C({\mathbb{R}}^2)$ such that $f\star\mu=g.$
\label{thm4}
\end{theorem}
We  prove the above theorem using  the following lemma.

\begin{lemma}
Let $C$ be a pointed convex cone in ${\mathbb{R}}^2$ with $v=(v_{1},v_{2})$ as the unit vector along the bisector of $C.$
Then for $\mu\in M_{d}({\mathbb{R}}^2)$ and $g\in C^{2*}({\mathbb{R}}^2),$ the following hold:
\begin{enumerate}
\item  If $supp(g)\subset H_{r}(v,\alpha)$ for some $\alpha\in {\mathbb{R}}$ and $supp(\mu)\subset C\cap H_{r}(v,\beta)$ for some $\beta>0$, then  there exits $f\in C^{2*}({\mathbb{R}}^2)$ such that $f\star(\delta_{0}+\mu)=g$ and $supp(f)\subset H_{r}(v,\alpha')$ for some $\alpha'\in {\mathbb{R}}.$

\item  If $supp(g)\subset H_{l}(v,\alpha)$ for some $\alpha\in {\mathbb{R}}$ and $supp(\mu)\subset (-C)\cap H_{l}(v,-\beta)$ for some $\beta>0$, then  there exits $f\in C^{2*}({\mathbb{R}}^2)$ such that $f\star(\delta_{0}+\mu)=g$ and $supp(f)\subset H_{l}(v,\alpha')$ for some $\alpha'\in {\mathbb{R}}.$

 \end{enumerate}
\label{lemma3.1}
\end{lemma}
\begin{proof}1. We denote the m-fold convolution of $\mu$ with itself by
$\mu^m$. As  $supp(\mu)\subset C\cap H_{r}(v,\beta),$ we have
$supp(\mu^n)\subset C\cap H_{r}(v,n\beta).$

\noindent Now
\begin{eqnarray*}
\left(\frac{\partial^2 g}{\partial x \partial y}\star
\mu^n\right)(x,y)&=&\int_{{\mathbb{R}}^2} \frac{\partial^2 g}{\partial x \partial y}(x-s,y-t)d\mu^{n}\\
          &=& \int_{C\cap H_{r}(v,n\beta)} \frac{\partial^2 g}{\partial x \partial y}(x-s,y-t)d\mu^{n}.
\end{eqnarray*}
Also, if  $(s,t)\in H_{r}(v,n\beta),$ then  $(x-s)v_{1}+(y-t)v_{2}\le xv_{1}+yv_{2}-n\beta<\alpha$   for sufficiently large $n$. Therefore, for
every $(x,y)\in {\mathbb{R}}^2,$ we have $(\frac{\partial^2 g}{\partial x \partial y}\star\mu^n)(x,y)=0$ for $n$ suitably large  and for every $n,$  $(\frac{\partial^2 g}{\partial x \partial y}\star\mu^n)(x,y)=0$ for all $(x,y)\in H_{l}(v, \alpha+n\beta).$

\noindent We  define
\begin{equation}
f(x,y):=\left(
g(x,y)+\sum_{k=1}^{\infty} (-1)^k\left(g\star\mu^k\right)\right)(x,y).
\label{thm1eqn1}
\end{equation}

\noindent
Let us consider  the  following partial
sums:

$$s_{n}(x,y)= \left(
g(x,y)+\sum_{k=1}^{n} (-1)^k(g\star\mu^k)\right)(x,y).$$

\noindent Then

$$\frac{\partial^2 s_{n}}{\partial x \partial y}(x,y)= \left(
\frac{\partial^2 g}{\partial x \partial y}(x,y)+\sum_{k=1}^{n} (-1)^k\left(\frac{\partial^2 g}{\partial x \partial y}\star\mu^k\right)\right)(x,y).$$
We show that the above sequence converges uniformly on every
compact set of ${\mathbb{R}}^2$. Let K be a compact subset of ${\mathbb{R}}^2.$ Then
$K\subset [a,b]\times [c,d]$ for some real numbers $a,b,c$ and $d$.

The function $(x,y)\mapsto xv_{1}+yv_{2}$ is  continuous  on ${\mathbb{R}}^2$ and hence it attains its maximum at some point $(x_{0},y_{0})$ on $K.$
Choose $N$
such that $x_{0}v_{1}+ y_{0}v_{2}-n\beta<\alpha$ for $n\ge N.$
 Then

$$ \frac{\partial^2 s_{n}}{\partial x \partial y}(x,y)- \frac{\partial^2 s_{k}}{\partial x \partial y}(x,y)=\left(\sum_{j=k+1}^n
(-1)^j\left(\frac{\partial^2 g}{\partial x \partial y}\star\mu^j\right)\right)(x,y)=0,$$ for $n\ge k\ge N.$

\noindent
Thus the sequence of functions
$\displaystyle{\left\{\frac{\partial^2 s_{k}}{\partial x \partial y}(x,y)\right\}}$ is uniformly Cauchy on every compact set and hence converges uniformly on every compact subset of ${\mathbb{R}}^2$.

As the partial sums  $s_{k}(x,y)$ and their derivatives $\frac{\partial^2 s_{k}}{\partial x \partial y}(x,y)$ of  the  series \ref{thm1eqn1} converge uniformly on compact sets,  we get
$$\frac{\partial^2 f}{\partial x \partial y}(x,y):=\left(
\frac{\partial^2 g}{\partial x \partial y}(x,y)+\sum_{n=1}^{\infty} (-1)^n\left(\frac{\partial^2 g}{\partial x \partial y}\star\mu^n\right)\right )(x,y).$$

\noindent Therefore $\frac{\partial^2 f}{\partial x \partial y}$ is continuous and hence $ f\in C^{2*}({\mathbb{R}}^2).$
 It is very easy to check  that  $$f\star(\delta_0+\nu)=g $$ and hence $f\star\mu=g.$ From the definition of $f$ it is simple to verify that $supp(f)\subset H_r(v,\alpha')$ for some $\alpha'\in {\mathbb{R}}$.

\vskip 1em
 The other  case  can also be proved using  similar arguments.
\end{proof}
\hfill{$\square$}

First we shall prove a special case of theorem 3.1.

\begin{theorem}
Let  $\mu\in M_{c}({\mathbb{R}}^2)$ be  absolutely continuous with a density function in the form $\alpha_{1}\chi_{R_{1}}+\alpha_{2}\chi_{R_{2}}$, where $R_{i}=\left[a_{i}, b_{i}\right]\times \left[c_{i}, d_{i}\right]$  satisfying $\alpha_{1},\alpha_{2}\neq 0$ and set conditions (\ref{cond1}) or (\ref{cond2}). Then for every
 $g\in C^{2*}({\mathbb{R}}^2),$ there exists $f\in C^{2*}({\mathbb{R}}^{2})$ such that $f\star\mu=g.$
 \label{thm2}
\end{theorem}
\begin{proof}

We can write  the density function of $\mu$ as $$\alpha_{1}\chi_{[-S_{1},S_{1}]\times [-T_{1},T_{1}]}\star \delta_{(M_{1},N_{1})}+\alpha_{2}\chi_{[-S_{2},S_{2}]\times [-T_{2}, T_{2}]}\star \delta_{(M_{2},N_{2})}.$$
Suppose that  $a_{1}<a_{2},b_{1}<b_{2},c_{1}<c_{2}, d_{1}<d_{2}.$

\noindent Consider the following linear partial  difference equations with constant coefficients:
\begin{eqnarray*}
\frac{1}{\alpha_{1}}\left[\sum_{j,k=0}^{\infty}g_{1}\left(x+M_{1}+(2j+1)S_{1}, y+N_{1}+(2k+1)T_{1} \right)\right ]
&+& \frac{1}{\alpha_{2}}\left[\sum_{j,k=0}^{\infty}g_{1}\left(x+M_{2}+(2j+1)S_{2},y+N_{2}+(2k+1)T_{2}\right)\right ] \\ &=&\frac{1}{\alpha_{2}}\left[\sum_{j,k=0}^{\infty}g\left(x+M_{2}+(2j+1)S_{2},y+N_{2}+(2k+1)T_{2}\right)\right ].\\
\label{thm2eqn1}
\end{eqnarray*}

\begin{eqnarray*}
\frac{1}{\alpha_{1}}\left[\sum_{j,k=0}^{\infty}g_{1}\left(x+M_{1}-(2j+1)S_{1}, y+N_{1}-(2k+1)T_{1} \right)\right ]
&+& \frac{1}{\alpha_{2}}\left[\sum_{j,k=0}^{\infty}g_{1}\left(x+M_{2}-(2j+1)S_{2},y+N_{2}-(2k+1)T_{2}\right)\right ] \\ &=&\frac{1}{\alpha_{2}}\left[\sum_{j,k=0}^{\infty}g\left(x+M_{2}-(2j+1)S_{2},y+N_{2}-(2k+1)T_{2}\right)\right ].\\
\label{thm2eqn12}
\end{eqnarray*}

\noindent We note that the above equations make sense only for a suitable class of functions $g.$ These equations can be written as
\begin{equation}
g_{1}\star\nu_{1}=h_{1},
\label{thm2eqn2}
\end{equation}
\begin{equation}
g_{2}\star\nu_{2}=h_{2},
\label{thm2eqn21}
\end{equation}

where
\begin{eqnarray}
\nu_{1}&=&\frac{1}{\alpha_{1}}\left[\sum_{j,k=0}^{\infty}\delta_{(-M_{1}-(2j+1)S_{1}, -N_{1}-(2k+1)T_{1})}\right ] + \frac{1}{\alpha_{2}}\left[\sum_{j,k=0}^{\infty}\delta_{(-M_{2}-(2j+1)S_{2}, -N_{2}-(2k+1)T_{2} }\right ],
\\
\nu_{2}&=&\frac{1}{\alpha_{1}}\left[\sum_{j,k=0}^{\infty}\delta_{(-M_{1}+(2j+1)S_{1}, -N_{1}+(2k+1)T_{1})}\right ] + \frac{1}{\alpha_{2}}\left[\sum_{j,k=0}^{\infty}\delta_{(-M_{2}+(2j+1)S_{2}, -N_{2}+(2k+1)T_{2} }\right ],
\\
h_{1}(x,y) &=& \frac{1}{\alpha_{2}}\left[\sum_{j,k=0}^{\infty}g\left(x+M_{2}+(2j+1)S_{2}, y+N_{2}+(2k+1)T_{2}\right)\right ],\\
h_{2}(x,y) &=& \frac{1}{\alpha_{2}}\left[\sum_{j,k=0}^{\infty}g\left(x+M_{2}-(2j+1)S_{2}, y+N_{2}-(2k+1)T_{2}\right)\right ].
\end{eqnarray}

\noindent Clearly $\nu_{1},\nu_{2}\in M_{d}({\mathbb{R}}^{2})$ and  $supp(\nu_{1})\subset Q_{3}(\alpha^{**},\beta^{**})$ and $supp(\nu_{2})\subset Q_{1}(\alpha^{*},\beta^{*})$, where $\alpha^{**}=max\{-M_{1}-S_{1},-M_{2}-S_{2}\},$ $\beta^{**}= max\{-N_{1}-T_{1},-N_{2}-T_{2}\},$ $\alpha^{*}=min\{-M_{1}+S_{1},-M_{2}+S_{2}\}$  and $\beta^{*}=min\{-N_{1}+T_{1},-N_{2}+T_{2}\}.$

We can write
\begin{eqnarray*}
h_{1}(x,y)&=&(g\star\psi_{1})(x,y)\\
h_{2}(x,y)&=&(g\star\psi_{2})(x,y), \mbox{ where}\\
\psi_{1}&=&\frac{1}{\alpha_{2}}\sum_{j,k=0}^{\infty}\delta_{(-M_{2}-(2j+1)S_{2}, -N_{2}-(2k+1)T_{2}) },\\
\psi_{2}&=&\frac{1}{\alpha_{2}}\sum_{j,k=0}^{\infty}\delta_{(-M_{2}+(2j+1)S_{2}, -N_{2}+(2k+1)T_{2} )}.
\end{eqnarray*}

It may be noted that some subtle combinations might lead to   $\nu_{1}=0$ and  $\nu_{2}=0.$
We shall  show that under the assumptions in the hypothesis, $\nu_{1},\nu_{2}\neq 0.$ We note that $$supp\left(\frac{1}{\alpha_{1}}\left[\sum_{j,k=0}^{\infty}\delta_{(-M_{1}-(2j+1)S_{1},-N_{1}-(2k+1)T_{1})}\right ]\right)\subset Q_{3}(-M_{1}-S_{1},-N_{1}-T_{1})$$ and
$$supp\left(\frac{1}{\alpha_{2}}\left[\sum_{j,k=0}^{\infty}\delta_{(-M_{2}-(2j+1)S_{2},-N_{2}-(2k+1)T_{2})}\right ]\right)\subset Q_{3}(-M_{2}-S_{2},-N_{2}-T_{2}).$$

\noindent
 As $b_{1}<b_{2},$ we have $\max\{-M_{1}-S_{1}, -M_{2}-S_{2} \}=-M_{1}-S_{1}=-b_{1}=m(\mbox{say})$ and hence
 we can find  $\epsilon_{1}>0$ such that $$-M_{1}-S_{1}\in (m-\epsilon_{1},m+\epsilon_{1}),\;\; -M_{2}-S_{2}\not\in (m-\epsilon_{1},m+ \epsilon_{1})\mbox{ and } -M_{1}-2S_{1}<m-\epsilon_{1}.$$
  \noindent
  Since $d_{1}<d_{2},$ $ \max\{-N_{1}-T_{1}, -N_{2}-T_{2} \}=-N_{1}-T_{1}=-d_{1}=n(\mbox{say})$. Hence as above we can find $\epsilon_{2}>0$ such that $$-N_{1}-T_{1}\in (n-\epsilon_{2},n+\epsilon_{2}),\;\; -N_{2}-T_{2}\not\in (n-\epsilon_{2},n+ \epsilon_{2})\mbox{ and } -N_{1}-2T_{1}<n-\epsilon_{2}.$$
  Choose a continuous function $\phi$ such that $\phi(x,y)=0 $ for $(x,y)$ in the complement of the set  $(m-\epsilon_{1},m+\epsilon_{1})\times (n-\epsilon_{2},n+\epsilon_{2})$  and $\phi(m,n)=1.$ Then $\int_{{\mathbb{R}}^2}\phi(x,y)d\nu_{1}=  \frac{1}{\alpha_{1}}$ and
 therefore $\nu_{1}\neq 0.$

 By similar arguments, using the conditions $a_{1}<a_{2}, c_{1}<c_{2},$ we can prove that $\nu_{2}\neq 0.$

\noindent
 Using lemma \ref{lemma2.5}, there is a pointed convex cone $C$ with a bisector $v$ and $(x^{**},y^{**})\in \mbox{supp}(\nu_{1}), (x^{*},y^{*})\in \mbox{supp}(\nu_{2})$ such that $\nu_{1}=c^{**}\delta_{(x^{**},y^{**})}\star(\delta_{0} +\eta_{1})$ and $\nu_{2}=c^{*}\delta_{(x^{*},y^{*})}\star(\delta_{0} +\eta_{2})$ and $\mbox{supp}(\eta_{1})\subset H_{l}(v,-\beta)\cap (-C)$ and $\mbox{supp}(\eta_{2})\subset H_{r}(v,\alpha)\cap C$ with $\alpha, \beta>0$ and $C\subset H_{r}(v,0).$

 In view of lemma \ref{lemma2.4}, it suffices  to prove the theorem for the cases: $supp(g)\subset H_{r}(v,a)$ and   $supp(g)\subset H_{l}(v,b).$

 \vskip 2em
 \noindent
  Case(i): Suppose that $supp(g)\subset H_{l}(v,b).$

 \vskip 1em
 In this case equation \ref{thm2eqn2} makes sense and $h_{1}\in C^{2*}({\mathbb{R}}^2)$ with $supp(h_{1})\subset H_{l}(v,\alpha')$ for some $\alpha'\in {\mathbb{R}}.$

 By lemma \ref{lemma3.1}, equation \ref{thm2eqn2} has a solution $g_{1}\in C^{2*}({\mathbb{R}}^2)$ such that $supp(g_{1})\subset H_{l}(v,\alpha'')$ for some $\alpha''\in {\mathbb{R}}.$  We use this $g_{1}$ to construct a solution $f.$

\noindent
Define $$f(x,y)=\left(\frac{\partial^{2}g_{1}}{\partial x\partial y}\star \left(\frac{1}{\alpha_{1}}\sum_{j,k=0}^{\infty}\delta_{(-M_{1}-(2j+1)S_{1},-N_{1}-(2k+1)T_{1})}\right)\right)(x,y).$$
Now
\begin{eqnarray*} \left (f\star \alpha_{1}\chi_{R_{1}} \right)(x,y)&=&\alpha_{1}\left(f\star \delta_{(M_{1},N_{1})}\star \chi_{[-S_{1},S_{1}]\times [-T_{1},T_{1}]}\right)(x,y)\\
          &=&\sum_{j,k=0}^{\infty}\int_{-S_{1}}^{S_{1}}\int_{-T_{1}}^{T_{1}}\frac{\partial^{2}g_{1}}{\partial x\partial y}(x+(2j+1)S_{1}-s,y+(2k+1)T_{1}-t )dtds\\
          &=&g_{1}(x,y).
\end{eqnarray*}
\noindent
Also differentiating equation \ref{thm2eqn2} partially with respect to $x$ and $y$, we get
\begin{equation}
\left(\frac{\partial^{2}g_{1}}{\partial x\partial y}\star\nu_{1}\right)(x,y)=\frac{\partial^{2}h_{1}}{\partial x\partial y}(x,y)\\
\end{equation}
I.e., \begin{eqnarray*}
\left(\frac{\partial^{2}g_{1}}{\partial x\partial y}\star \left(\frac{1}{\alpha_{1}}\sum_{j=0}^{\infty}\delta_{(-M_{1}-(2j+1)S_{1},-N_{1}-(2k+1)T_{1})}\right)\right)(x,y) &+&\left(\frac{\partial^{2}g_{1}}{\partial x\partial y}\star \left (\frac{1}{\alpha_{2}}\sum_{j=0}^{\infty}\delta_{(-M_{2}-(2j+1)S_{2},-N_{2}-(2k+1)T_{2} )}\right)\right)(x,y)\\
 &=& (\frac{\partial^{2}g}{\partial x\partial y}\star\psi_{1})(x,y).\\
\end{eqnarray*}
\noindent The above equation may be rewritten as
\begin{equation}
f(x,y)=-\left(\frac{\partial^{2}g_{1}}{\partial x\partial y}\star \left(\frac{1}{\alpha_{2}}\sum_{j,k=0}^{\infty}\delta_{(-M_{2}-(2j+1)S_{2},-N_{2}-(2k+1)T_{2})}\right)\right)(x,y) +\left(\frac{\partial^{2}g}{\partial x\partial y}\star\psi_{1}\right)(x,y).
\label{thm2eqn3}
\end{equation}
\noindent Now using the representation of $f$ from equation \ref{thm2eqn3},
\begin{eqnarray*}
(f\star \alpha_{2} \chi_{R_{2}})(x,y)&=&\alpha_{2} (f\star \delta_{(M_{2},N_{2})}\star \chi_{[-S_{2},S_{2}]\times [-T_{2},T_{2}] })(x,y)\\
            &=& -\sum_{j,k=0}^{\infty}\int_{-S_{2}}^{S_{2}} \int_{-T_{2}}^{T_{2}} \frac{\partial^{2}}{\partial x\partial y}g_{1}(x+(2j+1)S_{2}-s, y+(2k+1)T_{2}-t )dtds\\
            &&\;\;\;\;\;\;\;\;\;+\sum_{j,k=0}^{\infty}\int_{-S_{2}}^{S_{2}} \int_{-T_{2}}^{T_{2}}\frac{\partial^{2}}{\partial x\partial y}g(x+(2j+1)S_{2}-s,y+(2k+1)T_{2}-t)dtds\\
            &=&-g_{1}(x,y)+g(x,y)\\
            &=& -(f\star \alpha_{1} \chi_{R_{1}})(x,y) +g(x,y).
\end{eqnarray*}

\noindent Hence $f\star\mu=g.$

\vskip 2em
\noindent Case(ii):  Suppose that $supp(g)\subset H_{r}(v,a).$

\vskip 1em
In this case we use  equation \ref{thm2eqn21} instead of equation \ref{thm2eqn2} as $h_{2}\in C^{2*}({\mathbb{R}}^2)$ with $supp(h_{2})\subset H_{r}(v,\alpha')$ for some $\alpha'\in {\mathbb{R}}.$
 By lemma \ref{lemma3.1}, equation \ref{thm2eqn21} has a solution $g_{2}\in C^{2*}({\mathbb{R}}^2)$ such that $supp(g_{2})\subset H_{r}(v,\alpha'')$ for some $\alpha''\in {\mathbb{R}}.$
 Using  this $g_{2}$ we can construct a funtion $f$ as in case (i) such that $f\star \mu=g.$

 The case    $a_{2}<a_{1},b_{2}<b_{1},c_{1}<c_{2}, d_{1}<d_{2}$ can be analysed as the other case by considering the following difference equations:
\begin{eqnarray*}
\frac{1}{\alpha_{1}}\left[\sum_{j,k=0}^{\infty}g_{1}\left(x+M_{1}-(2j+1)S_{1}, y+N_{1}+(2k+1)T_{1} \right)\right ]
&+& \frac{1}{\alpha_{2}}\left[\sum_{j,k=0}^{\infty}g_{1}\left(x+M_{2}-(2j+1)S_{2},y+N_{2}+(2k+1)T_{2}\right)\right ] \\ &=&\frac{1}{\alpha_{2}}\left[\sum_{j,k=0}^{\infty}g\left(x+M_{2}-(2j+1)S_{2},y+N_{2}+(2k+1)T_{2}\right)\right ].\\
\label{thm2eqn1}
\end{eqnarray*}

\begin{eqnarray*}
\frac{1}{\alpha_{1}}\left[\sum_{j,k=0}^{\infty}g_{1}\left(x+M_{1}+(2j+1)S_{1}, y+N_{1}-(2k+1)T_{1} \right)\right ]
&+& \frac{1}{\alpha_{2}}\left[\sum_{j,k=0}^{\infty}g_{1}\left(x+M_{2}+(2j+1)S_{2},y+N_{2}-(2k+1)T_{2}\right)\right ] \\ &=&\frac{1}{\alpha_{2}}\left[\sum_{j,k=0}^{\infty}g\left(x+M_{2}+(2j+1)S_{2},y+N_{2}-(2k+1)T_{2}\right)\right ].\\
\label{thm2eqn12}
\end{eqnarray*}
As the proof is similar to the other case we omit the proof.

 \end{proof}

\hfil{$\Box$}

\noindent {\bf Proof of Theorem 3.1}
 \begin{proof}

 Let $\sum_{k=1}^{n} \alpha_{k}\chi_{R_{k}}$ be  the density function of $\mu$ such that $n\ge 2.$

\noindent
 As $n=2$  is  considered in theorem \ref{thm2}, we assume that $n>2.$

\noindent
 The density function of $\mu$ can be rewritten as $\sum_{k=1}^n \alpha_{k}\chi_{[-S_{k},S_{k}]\times [-T_{k},T_{k}]}\star \delta_{(M_{k},N_{k})}.$

Suppose that $ a_{i}<a_{i+1},b_{i}<b_{i+1}, c_{i}<c_{i+1}, d_{i}<d_{i+1}.$

\noindent We define the following discrete measures in $M_{d}({\mathbb{R}}^2)$:

 \begin{eqnarray*}
\psi &=&\frac{1}{\alpha_{n}}\sum_{i,j=0}^{\infty}\delta_{(-M_{n}-(2i+1)S_{n},-N_{n}-(2j+1)T_{n})},\\
\tilde{\psi} &=&\frac{1}{\alpha_{n}}\sum_{i,j=0}^{\infty}\delta_{(-M_{n}+(2i+1)S_{n},-N_{n}+(2j+1)T_{n})},\\
\mu_{1}&=&\frac{1}{\alpha_{1}}\sum_{i,j=0}^{\infty}\delta_{(-M_{1}-(2i+1)S_{1},-N_{1}-(2j+1)T_{1})}\mbox{ and }\\
\tilde{\mu}_{1}&=&\frac{1}{\alpha_{1}}\sum_{i,j=0}^{\infty}\delta_{(-M_{1}+(2i+1)S_{1},-N_{1}+(2j+1)T_{1})}.\\
\end{eqnarray*}
Also for $2\leq k\leq n-1,$ define
\begin{eqnarray*}
\zeta_{k}&=&\frac{\alpha_{k}}{\alpha_{1}}\sum_{i,j=0}^{\infty}\delta_{(-M_{1}-(2i+1)S_{1}+M_{k}-S_{k},-N_{1}-(2j+1)T_{1}+N_{k}-T_{k} }\\
\eta_{k}&=&\frac{\alpha_{k}}{\alpha_{1}}\sum_{i,j=0}^{\infty}\delta_{(-M_{1}-(2i+1)S_{1}+M_{k}-S_{k},-N_{1}-(2j+1)T_{1}+N_{k}+T_{k} }\\
\xi_{k}&=&  \frac{\alpha_{k}}{\alpha_{1}}\sum_{i,j=0}^{\infty}\delta_{(-M_{1}-(2i+1)S_{1}+M_{k}+S_{k},-N_{1}-(2j+1)T_{1}+N_{k}-T_{k} }\\
\psi_{k} &=&\frac{\alpha_{k}}{\alpha_{1}}\sum_{i,j=0}^{\infty}\delta_{(-M_{1}-(2i+1)S_{1}+M_{k}+S_{k},-N_{1}-(2j+1)T_{1}+N_{k}+T_{k} }\\
\tilde{\zeta}_{k}&=&\frac{\alpha_{k}}{\alpha_{1}}\sum_{i,j=0}^{\infty}\delta_{(-M_{1}+(2i+1)S_{1}+M_{k}-S_{k},-N_{1}+(2j+1)T_{1}+N_{k}-T_{k} }\\
\tilde{\eta}_{k}&=&\frac{\alpha_{k}}{\alpha_{1}}\sum_{i,j=0}^{\infty}\delta_{(-M_{1}+(2i+1)S_{1}+M_{k}-S_{k},-N_{1}+(2j+1)T_{1}+N_{k}+T_{k} }\\
\tilde{\xi}_{k}&=&  \frac{\alpha_{k}}{\alpha_{1}}\sum_{i,j=0}^{\infty}\delta_{(-M_{1}+(2i+1)S_{1}+M_{k}+S_{k},-N_{1}+(2j+1)T_{1}+N_{k}-T_{k} }\\
\tilde{\psi}_{k}&=& \frac{\alpha_{k}}{\alpha_{1}}\sum_{i,j=0}^{\infty}\delta_{(-M_{1}+(2i+1)S_{1}+M_{k}+S_{k},-N_{1}+(2j+1)T_{1}+N_{k}+T_{k} }\\
\end{eqnarray*}
Clearly $\psi,\tilde{\psi},\mu_{1},\tilde{\mu}_{1},\psi_{k},\tilde{\psi}_{k},\eta_{k},\tilde{\eta}_{k},\xi_{k},\tilde{\xi}_{k},
\zeta_{k},\tilde{\zeta}_{k}\in M_{d}({\mathbb{R}}^2)$ and
\newpage
\begin{eqnarray*}
supp(\psi)&\subset& Q_{3}(-M_{n}-S_{n},-N_{n}-T_{n}),\\
 supp(\mu_{1})&\subset& Q_{3}(-M_{1}-S_{1},-N_{1}-T_{1}),\\
 supp(\zeta_{k}\star \psi)&\subset & Q_{3}(-M_{1}-S_{1}-M_{n}-S_{n}+M_{k}-S_{k},-N_{1}-T_{1}-N_{n}-T_{n}+N_{k}-T_{k}),\\
 supp(\eta_{k}\star\psi)&\subset & Q_{3}(-M_{1}-S_{1}-M_{n}-S_{n}+M_{k}-S_{k},-N_{1}-T_{1}-N_{n}-T_{n}+N_{k}+T_{k}),\\
 supp(\xi_{k}\star\psi)&\subset &Q_{3}(-M_{1}-S_{1}-M_{n}-S_{n}+M_{k}+S_{k},-N_{1}-T_{1}-N_{n}-T_{n}+N_{k}-T_{k}),\\
supp(\psi_{k}\star\psi)&\subset &Q_{3}(-M_{1}-S_{1}-M_{n}-S_{n}+M_{k}+S_{k},-N_{1}-T_{1}-N_{n}-T_{n}+N_{k}+T_{k}),\\
supp(\tilde{\psi})&\subset& Q_{1}(-M_{n}+S_{n},-N_{n}+T_{n}),\\
 supp(\tilde{\mu_{1}})&\subset& Q_{1}(-M_{1}+S_{1},-N_{1}+T_{1}),\\
 supp(\tilde{\zeta_{k}}\star\tilde{\psi})&\subset & Q_{1}(-M_{1}+S_{1}-M_{n}+S_{n}+M_{k}-S_{k},-N_{1}+T_{1}-N_{n}+T_{n}+N_{k}-T_{k}),\\
 supp(\tilde{\eta_{k}}\star\tilde{\psi})&\subset & Q_{1}(-M_{1}+S_{1}-M_{n}+S_{n}+M_{k}-S_{k},-N_{1}+T_{1}-N_{n}+T_{n}+N_{k}+T_{k}),\\
 supp(\tilde{\xi_{k}}\star\tilde{\psi})&\subset &Q_{1}(-M_{1}+S_{1}-M_{n}+S_{n}+M_{k}+S_{k},-N_{1}+T_{1}-N_{n}+T_{n}+N_{k}-T_{k}),\\
supp(\tilde{\psi_{k}}\star\tilde{\psi})&\subset &Q_{1}(-M_{1}+S_{1}-M_{n}+S_{n}+M_{k}+S_{k},-N_{1}+T_{1}-N_{n}+T_{n}+N_{k}+T_{k}).\\
\end{eqnarray*}

\noindent
Consider the following convolution equations:
\begin{equation}
g_{1}\star \nu=h,
\label{thm3eqn1}
\end{equation}
\begin{equation}
\tilde{g}_{1}\star \tilde{\nu}=\tilde{h},
\label{thm3eqn1new}
\end{equation}
where \begin{eqnarray*}
\nu&=&\mu_{1}+\psi +\left(\sum_{k=2}^{n-1}\zeta_{k}\right)\star\psi-\left(\sum_{k=2}^{n-1}\eta_{k}\right)\star\psi -\left(\sum_{k=2}^{n-1}\xi_{k}\right)\star\psi +\left(\sum_{k=2}^{n-1}\psi_{k}\right)\star\psi,\\
  h&=&g\star\psi, \\
 \tilde{\nu}&=&\tilde{\mu}_{1} +\tilde{\psi} +\left(\sum_{k=2}^{n-1}\tilde{\zeta}_{k}\right)\star\tilde{\psi}-\left(\sum_{k=2}^{n-1}\tilde{\eta}_{k}\right)
  \star\tilde{\psi}-\left(\sum_{k=2}^{n-1}\tilde{\xi}_{k}\right)\star\tilde{\psi}
 +\left(\sum_{k=2}^{n-1}\tilde{\psi}_{k}\right)\star\tilde{\psi} \mbox{ and }\\
  \tilde{h}&=&\tilde{g}\star\tilde{\psi}. \\
 \end{eqnarray*}

 \noindent Then $supp(\nu)\subset Q_{3}(\rho,\theta)$ for some $\rho,\theta\in {\mathbb{R}}$ and $supp(\tilde{\nu})\subset Q_{1}(\tilde{\rho},\tilde{\theta})$ for some $\tilde{\rho},\tilde{\theta}\in {\mathbb{R}}.$

In order to solve equation \ref{thm3eqn1}, we need to make sure that $\nu\neq 0.$

\noindent
 Since  $b_{k}-b_{n}<0$  for every $k$,  we get $$-M_{1}-S_{1}-M_{n}-S_{n}+M_{k}\pm S_{k}<-M_{1}-S_{1} $$

\noindent
Choose $r_{1}>0$ such that $$-M_{1}-S_{1}-M_{n}-S_{n}+M_{k}\pm S_{k}<-M_{1}-S_{1}-r_{1}
 \mbox{ for } 2\le k\le n-1,$$
 and
 $$-M_{1}-2S_{1}\not\in (-M_{1}-S_{1}-r_{1},-M_{1}-S_{1}+r_{1}).$$

\noindent
 Similarly, as $d_{k}-d_{n}<0$ for $2\le k\le n-1,$ we get
$$-N_{1}-T_{1}-N_{n}-T_{n}+N_{k}\pm T_{k}<-N_{1}-T_{1} .$$

\noindent
Choose $r_{2}>0$ such that, for $2\le k\le n-1,$
$$-N_{1}-T_{1}-N_{n}-T_{n}+N_{k}\pm T_{k}<-N_{1}-T_{1}-r_{2},$$
and
$$-N_{1}-2T_{1}\not\in (-N_{1}-T_{1}-r_{2},-N_{1}-T_{1}+r_{2}).$$

\noindent
Therefore $$[-M_{1}-S_{1}-r_{1},-M_{1}-S_{1}+r_{1}]\times [-N_{1}-T_{1}-r_{2},-N_{1}-T_{1}+r_{2}]\cap supp(\nu)=\{(-M_{1}-S_{1}, -N_{1}-T_{1})\}.$$

\noindent
Choose an appropriate continuous function $\phi$ such that $\phi(x,y)=0$ for $(x,y) $ not in $[-M_{1}-S_{1}-r_{1},-M_{1}-S_{1}+r_{1}]\times [-N_{1}-T_{1}-r_{2},-N_{1}-T_{1}+r_{2}]$  and $\phi(-M_{1}-S_{1},-N_{1}-T_{1})=1.$ Then
$$\int_{{\mathbb{R}}^{2}}\phi(x,y)d\nu=\frac{1}{\alpha_{1}}\neq 0.$$ Hence $\nu\neq 0.$

Along the similar lines, by utilizing the conditions $a_{k}-a_{n}<0$ and $c_{k}-c_{n}<0$, we  can show that $\tilde{\nu}\neq 0.$

Using lemma \ref{lemma2.5}, there exists $(x^{**},y^{**})\in \mbox{supp}(\nu)$ and $(x^{*},y^{*})\in supp(\tilde{\nu})$ and a pointed convex cone $C$ with bisector $v$ such that $\nu=c^{**}(\delta_{0}+\nu_{1})\star \delta_{(x^{**},y^{**})}$ and
$\tilde{\nu}=c^{*} (\delta_{0}+\tilde{\nu}_{1})\star \delta_{(x^{*},y^{*})}$, $C\subset H_{r}(v,0)$,
$\mbox{supp}(\nu_{1})\subset H_{l}(v,-\rho)\cap(-C)$ and $\mbox{supp}(\tilde{\nu}_{1})\subset H_{r}(v,\theta)\cap C$  for some $\rho,\theta>0$ and
as a consequence of lemma \ref{lemma2.4}, it is sufficient to prove the result for each of the cases $\mbox{supp}(g)\subset H_{r}(v,\alpha)$ and  $\mbox{supp}(g)\subset H_{l}(v,\beta)$ separately.

As $g\in C^{2*}({\mathbb{R}}^2)$, by lemma \ref{lemma3.1},  $h\in C^{2*}({\mathbb{R}}^2)$ and $\mbox{supp}(h)\subset Q_{3}(\alpha'_{1},\beta'_{1})$ for some $\alpha'_{1},\beta'_{1}\in {\mathbb{R}}.$

\noindent {\bf Case(i):} Suppose that $\mbox{supp}(g)\subset H_{l}(v,\rho).$
\noindent
By lemma \ref{lemma3.1}, there exists a solution $g_{1}\in C^{2*}({\mathbb{R}}^{2})$ such that $g_{1}\star \nu=h$ and
$\mbox{supp}(g_{1})\subset Q_{3}(\rho',\theta')$ for some $\rho',\theta'\in {\mathbb{R}}.$ We use this $g_{1}$ to construct a solution $f$.

\vskip 1em
\noindent For $2\leq k\leq n-1,$ set
\begin{eqnarray*}
g_{k}&=&g_{1}\star (\zeta_{k}-\eta_{k}-\xi_{k}+\psi_{k}) \mbox{ and }\\
g_{n}&=&g-g_{1}-g_{2}-\cdots-g_{n-1}.
\end{eqnarray*}

\noindent Then $g_{2},g_{3},\ldots,g_{n}\in C^{2*}({\mathbb{R}}^{2})$ and $g=g_{1}+g_{2}+\cdots +g_{n}.$

\noindent Define
\begin{equation}
f=\frac{\partial^{2}g_{1}}{\partial x\partial y}\star\mu_{1}.
\label{thm3eqn2}
\end{equation}

\noindent Now
\begin{eqnarray*}\left( f\star \alpha_{1}\chi_{[a_{1},b_{1}]\times [c_{1},d_{1}]}\right)(x,y)&=& \left( f\star \alpha_{1}\delta_{(M_{1},N_{1})}\star\chi_{[-S_{1},S_{1}]}\right)(x,y)\\
&=&\left( \frac{\partial^{2}g_{1}}{\partial x\partial y}\star\left(\frac{1}{\alpha_{1}}\sum_{j,k=0}^{\infty}\delta_{(-M_{1}-(2j+1)S_{1},-N_{1}-(2k+1)T_{1})}\right )\star  \alpha_{1}\delta_{(M_{1},N_{1})}\star\chi_{[-S_{1},S_{1}]\times [-T_{1},T_{1}]}\right)(x,y)\\
&=&\left( \frac{\partial^{2}g_{1}}{\partial x\partial y }\star  \left ( \sum_{i,j=0}^{\infty}\delta_{(-(2i+1)S_{1},-(2j+1)T_{1} }\right )\star \chi_{[-S_{1},S_{1}]\times [-T_{1},T_{1}]}\right)(x,y)\\
&=&\sum_{i,j=0}^{\infty}\left(\int_{-S_{1}}^{S_{1}}\int_{-T_{1}}^{T_{1}} \frac{\partial^{2}}{\partial x\partial y} g_{1}(x+(2i+1)S_{1}-s,y+(2j+1)T_{1}-t )dtds\right )\\
&=&g_{1}(x,y).
\end{eqnarray*}

\noindent For $2\leq k \leq n-1,$
\begin{eqnarray*}
\left(f\star \alpha_{k}\chi_{[a_{k},b_{k}]\times [c_{k},d_{k}]}\right)(x,y)&=& \left(\frac{\partial^{2}}{\partial x\partial y }g_{1}\star\mu_{1}\star \alpha_{k}\delta_{(M_{k},N_{k})}\star\chi_{[-S_{k},S_{k}]\times [-T_{k},T_{k}]}\right)(x,y)\\
  &=&\left(\frac{\partial^{2}}{\partial x\partial y }g_{1}\star \left(\frac{\alpha_{k}}{\alpha_{1}}\sum_{i,j=0}^{\infty}\delta_{-M_{1}+M_{k}-(2i+1)S_{1}, -N_{1}+N_{k}-(2j+1)T_{1}}\right) \star\chi_{[-S_{k},S_{k}]\times [-T_{k},T_{k}]}\right)(x,y)\\
  &=&\frac{\alpha_{k}}{\alpha_{1}}\sum_{i,j=0}^{\infty}\int_{-S_{k}}^{S_{k}}\int_{-T_{k}}^{T_{k}}\frac{\partial^{2}}{\partial x\partial y }g_{1}(x+M_{1}-M_{k}+(2i+1)S_{1}-s,\\
 &&\hspace*{2in} y+N_{1}-N_{k}+(2j+1)T_{1}-t )dtds\\
 &=& \frac{\alpha_{k}}{\alpha_{1}}\sum_{i,j=0}^{\infty}g_{1}(x+M_{1}-M_{k}+(2i+1)S_{1}+S_{k},y+N_{1}-N_{k}+(2j+1)T_{1}+T_{k})\\
 &&\;\;\;\;\;\; +\frac{\alpha_{k}}{\alpha_{1}} \sum_{i,j=0}^{\infty}g_{1}(x+M_{1}-M_{k}+(2i+1)S_{1}-S_{k}, y+N_{1}-N_{k}+(2j+1)T_{1}-T_{k})\\
  &&\;\;\;\;\;\; -\frac{\alpha_{k}}{\alpha_{1}} \sum_{i,j=0}^{\infty}g_{1}(x+M_{1}-M_{k}+(2i+1)S_{1}-S_{k}, y+N_{1}-N_{k}+(2j+1)T_{1}+T_{k})\\
   &&\;\;\;\;\;\; -\frac{\alpha_{k}}{\alpha_{1}} \sum_{i,j=0}^{\infty}g_{1}(x+M_{1}-M_{k}+(2i+1)S_{1}+S_{k}, y+N_{1}-N_{k}+(2j+1)T_{1}-T_{k})\\
 &=&\left(g_{1}\star (\zeta_{k}-\eta_{k}-\xi_{k}+\psi_{k})\right)(x,y)\\
 &=&g_{k}(x,y).\\
\end{eqnarray*}
\noindent We observe that
\begin{eqnarray*}
g_{1}\star\nu=h & \Rightarrow
 &
g_{1}\star \mu_{1}+g_{1}\star\psi + g_{1}\star\sum_{k=2}^{n-1}\left(\zeta_{k}-\eta_{k}-\xi_{k}+\psi_{k}\right)\star\psi =g\star\psi\\
&\Rightarrow & g_{1}\star \mu_{1}+g_{1}\star\psi+g_{2}\star\psi+g_{3}\star\psi+\cdots + g_{n-1}\star\psi=g\star\psi\\
&\Rightarrow & g_{1}\star\mu_{1}=(g-g_{1}-g_{2}-\cdots -g_{n-1})\star \psi\\
&\Rightarrow & g_{1}\star\mu_{1}= g_{n}\star\psi.\\
\end{eqnarray*}

\noindent Differentiating partially, we get
\begin{equation}
\frac{\partial^{2}g_{1}}{\partial x\partial y }\star \mu_{1}=\frac{\partial^{2}g_{n}}{\partial x\partial y }\star\psi.
\label{thm3eqn3}
\end{equation}
\noindent Using equations \ref{thm3eqn2} and \ref{thm3eqn3} we obtain,  $$f(x,y)=\frac{1}{\alpha_{n}}\sum_{i,j=0}^{\infty}\frac{\partial^{2}g_{n}}{\partial x\partial y }(x+M_{n}+(2i+1)S_{n}, y+N_{n}+(2j+1)T_{n}).$$
\noindent Convolving both sides with $\alpha_{n}\delta_{(M_{n},N_{n})}\star \chi_{[-S_{n},S_{n}]\times [-T_{n},T_{n}]},$ we get
\begin{eqnarray*}
\left(f\star \alpha_{n}\delta_{(M_{n},N_{n})}\star \chi_{[-S_{n},S_{n}]\times [-T_{n},T_{n}]}\right)(x,y)
&=&\sum_{i,j=0}^{\infty}\int_{-S_{n}}^{S_{n}}\int_{-T_{n}}^{T_{n}}\frac{\partial^{2}g_{n}}{\partial x\partial y }(x+(2i+1)S_{n}-s,y+(2j+1)T_{n}-t )dtds \\
&=&g_{n}(x,y).\\
\end{eqnarray*}
Thus we have $g_{1}+ g_{2}+\cdots + g_{n}=g$ and $f\star \alpha_{i}\chi_{[a_{i},b_{i}]\times [c_{i},d_{i}]}=g_{i}$ for $1\le i\le n.$
Hence $f\star\mu=g.$

\vskip 1em
\noindent {\bf Case(ii):} The case  $supp(g)\subset H_{r}(v,\theta)$ can be proved by considering equation \ref{thm3eqn1new}.

The case  $a_{i+1}<a_{i}, b_{i+1}<b_{i},c_{i}<c_{i+1}, d_{i}<d_{i+1}$  can also be proved in a similar fashion.
\hfill{$\Box$}

\end{proof}

\noindent {\bf Open Problems:}

\begin{enumerate}
\item  Let the density function of $\mu$ be the indicator function of unit disc $B(0,1)={(x,y)\in \{\mathbb{R}}^2: x^2 +y^2 \le 1\}$  in ${\mathbb{R}}^2.$ What is the range of $C_{\mu}$?
\item Let  $R_{i}=\left[a_{i}, b_{i}\right]\times \left[c_{i}, d_{i}\right]$ be such that  none of the conditions (\ref{cond1}) and (\ref{cond2}) are  satisfied and   the density function of $\mu$ is  $\alpha_{1}\chi_{R_{1}}+\alpha_{2}\chi_{R_{2}},$ where $\alpha_{1},\alpha_{2}\neq 0.$ Is the range of $C_{\mu}$ equal to  $C^{2*}({\mathbb{R}}^2)?$
\end{enumerate}

\end{document}